\documentclass[12pt,leqno]{amsart}
\usepackage{enumerate}
\usepackage{amsrefs}
\usepackage{amsfonts, amsmath, amssymb, amscd, amsthm, bm, cancel}
\usepackage{url}
\usepackage{graphicx}
\usepackage[
linktocpage=true,colorlinks,citecolor=magenta,linkcolor=blue,urlcolor=magenta]{hyperref}
\usepackage{multicol}
\usepackage{comment}
\usepackage[margin=1in]{geometry}
\makeatletter
\@namedef{subjclassname@2020}{\textup{2020} Mathematics Subject Classification}
\makeatother
 \newtheorem{thm}{Theorem}[section]
 \newtheorem{Rem}{Remark}[section]
  \newtheorem{Cor}{Corollary}[section]

 \newtheorem{cor}[Cor]{Corollary}
  
 \newtheorem{lem}[thm]{Lemma}
 
 \newtheorem{prop}[thm]{Proposition}
 \theoremstyle{definition}
 
  \newtheorem*{ack}{Acknowledgments}

 \theoremstyle{remark}
 \newtheorem{rem}[Rem]{Remark} 

 \numberwithin{equation}{section}
 \allowdisplaybreaks

\newcommand{\tr}{\mathrm{tr}}

\begin{document}
\title[Spectral gaps and uniqueness for the isotropic $L_p$ Minkowski problem]{Spectral gaps and uniqueness in the isotropic $L_p$ Minkowski problem}

\author{Yao Wan}
\address{Department of Mathematics, The Chinese University of Hong Kong, Shatin, Hong Kong}
\email{\href{mailto:yaowan@cuhk.edu.hk}{yaowan@cuhk.edu.hk}}

\keywords{Hilbert-Brunn-Minkowski operator, spectral gap, uniqueness, $L_p$ Minkowski problem, local Brunn-Minkowski inequality}
\subjclass[2020]{53A07; 35A02; 52A20}


\begin{abstract}
We investigate the isotropic $L_p$ Minkowski problem in $\mathbb{R}^n$ for $p<-n$. We prove that a smooth strictly convex solution $K$ is the unit ball if the second non-zero eigenvalue of the Hilbert-Brunn-Minkowski operator satisfies $\lambda_2(-L_K)\ge\frac{-p-1}{n-1}$. We also derive a lower bound for $\lambda_2(-L_K)$ in terms of concentric inner and outer ellipsoids and give improved bounds for origin-centred and origin-symmetric bodies. Combined with a weighted centroid identity for the polar body, these estimates yield an explicit $\varepsilon_n>0$ for which the unit ball is the only solution when $-n-\varepsilon_n\le p<-n$. 
\end{abstract}

\maketitle

\section{Introduction}\label{sec:1}

A central problem in convex geometry is the Minkowski problem, which asks whether a given Borel measure on the unit sphere $\mathbb{S}^{n-1}$ arises as the surface area measure $S_K$ of a convex body $K$ in $\mathbb{R}^n$. In the smooth setting, this reduces to solving the Monge-Amp\`ere equation 
\begin{equation}\label{eq:MA}
    \det(\nabla^2 h_K + h_K I) = f \quad \text{on } \mathbb{S}^{n-1},
\end{equation}  
where $h_K$ is the support function of $K$ and $f$ is a given density. Existence and uniqueness (up to translation) were established by Minkowski \cites{Min1897, Min1903}, Aleksandrov \cite{Alek38}, and Fenchel and Jessen \cite{FJ38},  while regularity results were made by Lewy \cite{lewy38}, Nirenberg \cite{Nir53}, Cheng and Yau \cite{CY76}, Pogorelov \cite{Pog78}, and Caffarelli \cite{Caf90}, among others. For details, see Schneider's book \cite{Sch14}.

The Minkowski problem has been extended to the $L_p$ Minkowski problem, introduced by Lutwak \cites{Lut93, Lut96}, which generalizes the surface area measure to the $L_p$-surface-area measure $S_{p}K$. This leads to the Monge-Amp\`ere equation
\begin{equation}\label{eq:Lp-MA}
    h_K^{1-p} \det(\nabla^2 h_K + h_K I) = f \quad \text{on } \mathbb{S}^{n-1}.
\end{equation}
The case $p=1$ recovers the classical Minkowski problem, $p=0$ corresponds to the logarithmic Minkowski problem, and $p=-n$ corresponds to the centro-affine Minkowski problem. See \cites{BBCY19, BLYZ13, CW06, HLW16, HLYZ05, LYZ04} and the survey by B\"{o}r\"{o}czky \cite{Bor23}.

When $f$ is constant, the $L_p$ Minkowski problem (\ref{eq:Lp-MA}), known as the isotropic $L_p$ Minkowski problem, corresponds to the equation
\begin{equation}\label{is-p-M-eq}
    h_K^{1-p} \det(\nabla^2 h_K + h_K I) = 1 \quad \text{on } \mathbb{S}^{n-1}.
\end{equation}
This has attracted significant attention since Firey \cite{Fir74}. For $p=0$, the uniqueness of solutions to (\ref{is-p-M-eq}), known as the Firey conjecture, was resolved by Firey \cite{Fir74} for origin-symmetric cases, and later by Andrews \cites{And99,And03} for $n=2,3$, and Brendle, Choi and Daskalopoulos \cite{BCD17} for $n \geq 4$.
By the results of Andrews \cite{And99}, Andrews, Guan and Ni \cite{AGN16}, and Brendle, Choi and Daskalopoulos \cite{BCD17}, the unit ball is the only solution of \eqref{is-p-M-eq} for $-n<p<1$. The case $p=-n$ gives origin-centred ellipsoids having the same volume as the unit ball.
Recent progress includes stability results by Ivaki \cite{Iv22} and Hu-Ivaki \cite{HI2408}, novel approaches by Ivaki-Milman \cite{IM23} and Saroglou \cite{Sa22}, and uniqueness results for  $C^{\alpha}$-perturbations when $p \in [0,1)$ by Böröczky-Saroglou \cite{BS24}.

The supercritical case $p < -n$ remains particularly challenging and largely open. Guang, Li, and Wang \cite{GLW2203} proved the existence of a uniformly convex $C^{3,\gamma}$ solution to \eqref{eq:Lp-MA} for positive $C^{1,1}$ function $f$, for some $\gamma\in(0,1)$. However, when $n=2$, Du \cite{Du21} constructed a non-negative $C^{\alpha}$ function $f$, positive except at a pair of antipodal points, for which (\ref{eq:Lp-MA}) has no solution. For the isotropic case, when $n=2$, Andrews \cite{And03} provided a complete classification: if $-7 \leq p < -2$, the only solution to (\ref{is-p-M-eq}) is the unit circle; if $p < -7$, the only solutions are the unit circle and the curves $\Gamma_{k,p}$ with $k$-fold symmetry, for each integer $k$ satisfying $3 \leq k < \sqrt{2-p}$.
Du \cite{Du25} also studies uniqueness in the supercritical range, including results under regular-polytope symmetry.

Recently, Hu and Ivaki \cite{HI26} established a centro-affine Poincar\'e inequality for smooth strictly convex unconditional bodies, that is, bodies invariant under reflection in every coordinate hyperplane. In the normalization used below, their result gives $\lambda_2(-L_K)>\frac{n}{n-1}$ for such bodies. They also proved that the unit ball is the only smooth strictly convex unconditional solution to \eqref{is-p-M-eq} when $-n-1\le p<-n$ and $n\ge3$.

Andrews' planar classification and the higher-dimensional results under symmetry assumptions lead to the following question:

\begin{quote}
\emph{For each $n\ge3$, does there exist $\varepsilon_n>0$ such that the unit ball is the unique solution of \eqref{is-p-M-eq} whenever $-n-\varepsilon_n\le p<-n$, without symmetry assumption?}
\end{quote}

In this paper, we answer this question affirmatively and give an explicit choice of $\varepsilon_n$. Our approach combines a uniqueness criterion in terms of the spectral gap of the Hilbert-Brunn-Minkowski operator with a geometric lower bound for that gap.


Let $\mathcal{K}$ denote the class of convex bodies in $\mathbb{R}^n$ containing the origin in their interior, and let $\mathcal{K}_+^2\subset\mathcal{K}$ consist of those with $C^2$ boundary and positive Gauss curvature. We say that $K$ is \textit{origin-centred} if its centroid is at the origin, and \textit{origin-symmetric} if $K=-K$. Let $B$ denote the unit ball.
The Hilbert-Brunn-Minkowski operator originates in Hilbert's proof \cite{BF87} of the Brunn-Minkowski inequality and plays a central role in the work of Kolesnikov and Milman \cite{KM22} on the local $L_p$ Brunn-Minkowski conjecture. For $K\in\mathcal{K}_+^2$, it is given by
\begin{equation*}
    L_K z = \frac{1}{n-1}\tr\left((D^2h_K)^{-1}D^2(zh_K)\right)-z,
\end{equation*}
where $D^2h=\nabla^2h+hI$. The cone-volume measure of $K$ is
\[
    dV_K=\frac{1}{n}h_K\det(D^2h_K)\,d\mu,
\]
where $\mu$ is the spherical Lebesgue measure. The operator $-L_K$ has a non-negative self-adjoint realization on $L^2(V_K)$ with discrete spectrum. Its eigenspaces at $0$ and $1$ are, respectively,
\[
    E_0=\operatorname{span}\{1\},
    \qquad
    E_1^K=\left\{\ell_v^K(x)=\frac{\langle x,v\rangle}{h_K(x)}:v\in\mathbb{R}^n\right\}.
\]
We denote the next distinct eigenvalue by $\lambda_2(-L_K)>1$. If $K$ is origin-symmetric, we denote by $\lambda_{2,o}(-L_K)$ the first eigenvalue on the odd subspace orthogonal to $E_1^K$.

Our first result is a uniqueness criterion depending only on the spectral bound.

\begin{thm}\label{thm-unique}
    Let $n\ge3$, $p<-n$, and let $K\in\mathcal{K}_+^2$ be a solution to \eqref{is-p-M-eq}. If
    \begin{equation}\label{eq:threshold}
        \lambda_2(-L_K)\ge\frac{-p-1}{n-1},
    \end{equation}
    then $K$ is the unit ball. 
    For origin-symmetric $K$, $\lambda_2(-L_K)$ can be replaced by $\lambda_{2,o}(-L_K)$.
\end{thm}

\begin{rem}
    Note that at the unit ball,
\[
    \lambda_2(-L_B)=\frac{2n}{n-1},
    \qquad
    \lambda_{2,o}(-L_B)=\frac{3(n+1)}{n-1}.
\]
Thus the ball satisfies \eqref{eq:threshold} for $p\in [-2n-1,-n)$, and satisfies its odd counterpart for $p\in [-3n-4,-n)$.
\end{rem}


A key ingredient in the proof is the angular second moment identity
\begin{equation}\label{eq:angular-intro}
    \int_{\mathbb{S}^{n-1}}x\otimes x\,dV_K
    =\frac{V(K)}{n}I,
\end{equation}
valid for solutions of \eqref{is-p-M-eq} with $p\ne-n$; see Hu and Ivaki \cite[Lemma 4.2]{HI2408}.
Together with the divergence theorem, this identity shows that the inverse Gauss map $X_K=h_Kx+\nabla h_K$ and its radial part $h_Kx$ have the same componentwise orthogonal projection onto $E_1^K$. The resulting projection remainder also controls the non-zero mean of $X_K$. This allows us to apply the spectral gap without imposing a centroid or isotropy condition.

We next obtain a lower bound for $\lambda_2(-L_K)$ in terms of concentric inner and outer ellipsoids centred at the origin. Define
\begin{equation}\label{eq:geomconst}
     \kappa_0(K)=\inf\left\{\frac{R}{r}: rB\subset T(K)\subset RB,\; T\in\mathrm{GL}(n),\;0<r\le R\right\}.
\end{equation}
We also set
\begin{equation*}
    A_n=\frac{4n}{\pi^2}+\frac{4}{\pi},\qquad\eta_n=\frac{(n+2)/n}{25n+10\sqrt{n(n+2)}}.
\end{equation*}

\begin{thm}\label{thm-gap}
    Let $n\ge3$ and $K\in\mathcal{K}_+^2$. Then
    \begin{equation}\label{eq:geomgap}
        \lambda_2(-L_K)
        \ge
        1+\frac{1}{(n-1)A_n\kappa_0(K)^2}.
    \end{equation}
    If $K$ is origin-centred, we have $\kappa_0(K)\le n$ and
    \begin{equation}\label{eq:kls-gap}
        \lambda_2(-L_K)\ge 1+\frac{\eta_n}{n-1}.
    \end{equation}
    If $K$ is origin-symmetric, we have $\kappa_0(K)\le\sqrt{n}$.
\end{thm}

The proof follows the Neumann problem and Reilly formula approach of Kolesnikov and Milman \cite[Theorem 6.1]{KM22} and uses their trace estimate \cite[Theorem 6.6]{KM22}. To obtain the full spectral gap, we refine the Neumann problem argument by taking into account the contribution of constant vector fields. This allows the mean-gradient term to be removed without assuming evenness of the datum. Linear equivariance then allows optimization over the maps in \eqref{eq:geomconst}.
The estimate \eqref{eq:kls-gap} follows from the same argument in centroid and inertia position, using the dimension-free Poincar\'e bound \cite{BK26}. See also \cites{BKL26,SZ26} for the solution of the KLS conjecture. In particular, Theorems~\ref{thm-unique} and~\ref{thm-gap} imply that every non-ball solution to \eqref{is-p-M-eq} with $p<-n$ satisfies
\begin{equation*}
    (-p-n) A_n \kappa_0(K)^2>1.
\end{equation*}

Combining the spectral uniqueness criterion with the geometric spectral estimate, we obtain an explicit interval of uniqueness as follows.

\begin{thm}\label{thm-interval}
    Let $n\ge3$, and set
    \begin{equation}\label{eq:eps}
        \begin{split}
        b_n&=\frac{1}{n^2A_n}, \qquad d_n=(n+1)(1+\log(2n)),\\
        \varepsilon_n&=\max\left\{ \frac{4b_n}{\left(1+\sqrt{1+4 d_n b_n}\right)^2},\,\min\left\{\frac{1}{2d_n},\frac{\eta_n}{2n+1}\right\}\right\}.
        \end{split}
    \end{equation}
    For $-n-\varepsilon_n\le p<-n$, the unit ball is the unique solution to \eqref{is-p-M-eq} in $\mathcal{K}_+^2$. For origin-centred solutions, the conclusion holds with $\max\{b_n,\eta_n\}$ in place of $\varepsilon_n$. For origin-symmetric solutions, it holds with $\max\{nb_n,\eta_n\}$ in place of $\varepsilon_n$.
\end{thm}
\begin{rem}
Very recently, the uniqueness near \(p=-n\) from below has also been obtained independently by Liu-Lu \cite{LiuLu} and Rajpal \cite{Rajpal} through different approaches.
\end{rem}

For origin-centred and origin-symmetric bodies, the corresponding conclusions follow directly from Theorem~\ref{thm-gap}.
For a general solution, denote $p=-n-\varepsilon$. Equation \eqref{is-p-M-eq} and the divergence theorem yield the weighted centroid identity
\begin{equation}\label{eq:polar-centroid-intro}
    \int_{K^\circ}y|y|^\varepsilon\,dy=0,
\end{equation}
where $K^\circ$ is the polar body with respect to the origin. This identity controls the position of the origin relative to the centroid of a suitable linear image of $K^\circ$, and gives the estimate
\begin{equation*}
    \kappa_0(K)\le\frac{n}{1-d_n\varepsilon}
\end{equation*}
whenever $d_n\varepsilon<1$. Together with \eqref{eq:geomgap}, this gives the first width in \eqref{eq:eps}. When $d_n\varepsilon\leq \frac{1}{2}$, the same estimate gives $-K\subset 2nK$. This allows us to compare the spectral gaps at the origin and the centroid by Lemma~\ref{lem:origin-gap} and apply \eqref{eq:kls-gap} to obtain the second width in \eqref{eq:eps}.

As an application, we obtain a rigidity result for homothetically shrinking solutions of the $\alpha$-Gauss curvature flow with speed $\mathcal{K}^\alpha$, where $\mathcal{K}$ denotes the Gauss curvature.  At and above the affine exponent $\alpha=\frac{1}{n+1}$, the flow has been studied in e.g. \cites{And96,AGN16,BCD17}. After placing the centre of homothety at the origin and rescaling, the shrinker equation reduces to \eqref{is-p-M-eq} with $p=1-1/\alpha$. Theorem \ref{thm-interval} therefore yields the following.

\begin{cor}\label{cor-flow}
    Let $n\ge3$. A smooth closed uniformly convex homothetically shrinking solution of the $\alpha$-Gauss curvature flow is round if
    \[
        \frac{1}{n+1+\varepsilon_n}\leq\alpha<\frac{1}{n+1}.
    \]
    If its centre of homothety coincides with its centroid, one may replace $\varepsilon_n$ by $\max\{b_n,\eta_n\}$. If it is centrally symmetric about its centre of homothety, one may replace $\varepsilon_n$ by $\max\{nb_n,\eta_n\}$.
\end{cor}

This paper is organized as follows.
In Section~\ref{sec:2}, we recall the background and facts on convex bodies, cone-volume measures, and the Hilbert-Brunn-Minkowski operator.
In Section~\ref{sec:3}, we prove the uniqueness criterion under the spectral condition of Theorem \ref{thm-unique}.
In Section~\ref{sec:gap}, we establish the geometric lower bound for the spectral gap of Theorem \ref{thm-gap}.
Finally, in Section~\ref{sec:interval}, we prove Theorem \ref{thm-interval}, which gives an explicit interval of uniqueness.

\begin{ack}
    The author would like to thank Prof. Haizhong Li for introducing the author to this subject and for his continued interest and encouragement.
    This work was supported by Hong Kong RGC grants No. 14300623, No. 14304225, and No. 14308325.
\end{ack}

\section{Preliminaries}\label{sec:2}

\subsection{Convex bodies}$\ $

Let $(\mathbb{R}^{n}, \delta, \bar{\nabla})$ denote the Euclidean space with its standard inner product $\delta = \langle \cdot, \cdot \rangle$ and flat connection $\bar{\nabla}$. Let $(\mathbb{S}^{n-1}, g_0, \nabla)$ denote the unit sphere equipped with its round metric $g_0$ and Levi-Civita connection $\nabla$. The spherical Lebesgue measure on $\mathbb{S}^{n-1}$ is denoted by $\mu$. Unless otherwise specified, all integrals over the sphere are taken over $\mathbb{S}^{n-1}$.

Given a local orthonormal frame $\{e_1, \ldots, e_{n-1}\}$ on $\mathbb{S}^{n-1}$, the covariant derivatives of $h \in C^2(\mathbb{S}^{n-1})$ are denoted by $h_i := \nabla_{e_i} h$ and $h_{ij} := \nabla^2_{e_i,e_j} h$, respectively. Extending $h$ to a 1-homogeneous function on $\mathbb{R}^n\setminus\{0\}$, the restricted Hessian $D^2 h$ on $T\mathbb{S}^{n-1}$ is given by
\begin{align*}
    (D^2 h)_{ij} = \bar{\nabla}_{e_i,e_j}^2 h = h_{ij} + h \delta_{ij}, \quad  i,j=1,\ldots, n-1.
\end{align*}
We identify $g_0$ with the identity in an orthonormal frame and set
\begin{align*}
    C_h^2(\mathbb{S}^{n-1})=\{h\in C^2(\mathbb{S}^{n-1}):\ h>0,\ D^2h>0\}.
\end{align*}

A convex body in $\mathbb{R}^n$ is a compact convex set with non-empty interior. Let $\mathcal{K}$ denote the class of convex bodies in $\mathbb{R}^n$ containing the origin in their interior, with $\mathcal{K}_+^m$ referring to those with $C^m$ smooth boundaries and strictly positive Gaussian curvature. 
We also denote by $\mathcal{K}_{+,e}^2$ the subset of origin-symmetric bodies in $\mathcal{K}_+^2$, and by $C_{e}^2(\mathbb{S}^{n-1})$ and $C_o^2(\mathbb{S}^{n-1})$ the spaces of even and odd $C^2$ functions, respectively. Let $B$ denote the unit ball.

For a convex body $K \in \mathcal{K}$, its support function $h_K: \mathbb{S}^{n-1} \to \mathbb{R}$ is defined as
\begin{align*}
    h_K(x) := \max_{y \in K} \langle x, y \rangle, \quad x \in \mathbb{S}^{n-1}.
\end{align*}
For $K\in\mathcal K_+^2$, the Gauss map $\nu_K:\partial K\to\mathbb{S}^{n-1}$ is invertible, and its inverse is
\begin{align*}
    X_K(x) = h_K(x)x + \nabla h_K(x).
\end{align*}
The class of support functions of bodies in $\mathcal K_+^2$ coincides with $C_h^2(\mathbb{S}^{n-1})$; see \cite[Section 2.5]{Sch14}.
The principal curvatures of $\partial K$ at $X_K(x)$ are the eigenvalues of $(D^2h_K(x))^{-1}$. In particular, its Gauss curvature $\mathcal K_K$ satisfies
\begin{align*}
    \mathcal K_K(X_K(x))^{-1}=\det(D^2h_K(x)).
\end{align*}
The surface area measure and the $L_p$ surface area measure of $K\in\mathcal K_+^2$ are given by
\[
    dS_K=\det(D^2h_K)\,d\mu,
    \qquad
    dS_pK=h_K^{1-p}\,dS_K.
\]
The cone-volume measure is
\begin{equation*}
    dV_K=\frac1n h_K\,dS_K=\frac1n\,dS_0K,
    \qquad
    \int_{\mathbb{S}^{n-1}}dV_K=V(K),
\end{equation*}
where $V(K)$ denotes the volume of $K$. We also use the polar body
\[
    K^\circ=\left\{y\in\mathbb{R}^n: \langle y,z\rangle\le1 \text{ for all }z\in K \right\}.
\]
Writing $\rho_K(x)=\max\{t\ge0:tx\in K\}$ for the radial function, we have
\[
\rho_{K^\circ}(x)=\frac1{h_K(x)}.
\]
For further background, see \cites{Sch14, KM22, ACGL20}.

\subsection{Hilbert-Brunn-Minkowski operator}$\ $

Given $K\in\mathcal{K}_{+}^2$, the Hilbert-Brunn-Minkowski operator of $K$, denoted $L_K$, is the second-order linear differential operator on $C^2(\mathbb{S}^{n-1})$ defined by
\begin{equation}\label{eq:LK}
    L_K:=\tilde{L}_K-\text{Id},\ \, \widetilde L_Kz =\frac{1}{n-1}\tr\left( (D^2h_K)^{-1}D^2(zh_K)\right).
\end{equation}
This agrees with the mixed discriminant definition of Kolesnikov and Milman \cite{KM22}, introduced in their study of the local $L_p$ Brunn-Minkowski inequality. Up to normalization, the operator already appears in Hilbert's proof of the Brunn-Minkowski inequality; see \cite[Section 52]{BF87}. The operator $\Delta_K=(n-1)L_K$ can also be viewed as the centro-affine Laplacian; see \cites{Mi25,IM24,Mi24}.


In local coordinates, \eqref{eq:LK} gives
\begin{equation}\label{eq:LK-coordinate}
\begin{split}
      L_Kz &=\frac{1}{n-1}\big((D^2h_K)^{-1}\big)^{ij}\big(h_Kz_{ij}+(h_K)_iz_j+(h_K)_jz_i\big)\\
    &=\frac{1}{(n-1)h_K}\big((D^2h_K)^{-1}\big)^{ij}(h_K^2z_i)_j,
\end{split}
\end{equation}
where repeated indices are summed from $1$ to $n-1$.

The cofactor tensor $\det(D^2h_K)(D^2h_K)^{-1}$ is divergence-free for smooth support functions. Integration by parts, followed by approximation for $K\in\mathcal K_+^2$, therefore gives
\begin{equation}\label{self-ad}
\begin{split}
    \mathcal E_K(z_1,z_2) &:= \int z_1(-L_Kz_2)\,dV_K\\
    &=\frac{1}{n-1}\int h_K \left\langle (D^2h_K)^{-1}\nabla z_1, \nabla z_2\right\rangle dV_K\\
    &=\int z_2(-L_Kz_1)\,dV_K.
\end{split}
\end{equation}
The form on the right extends to $H^1(\mathbb{S}^{n-1})$.

We recall the following spectral properties of the Hilbert-Brunn-Minkowski operator.

\begin{thm}[\cite{KM22}]\label{thm-LK}
    Let $K\in\mathcal K_+^2$.
    \begin{enumerate}[(1)]
        \item The operator $-L_K$ is symmetric and non-negative on $L^2(V_K)$ and admits a self-adjoint extension associated with the energy form $\mathcal E_K$. Its spectrum is discrete, with eigenvalues of finite multiplicity tending to infinity.

        \item The eigenspace at $0$ is
        \[
            E_0=\operatorname{span}\{1\},
        \]
        and the first non-zero eigenvalue is $1$, with eigenspace
        \[
            E_1^K=\left\{\ell_v^K(x)=\frac{\langle x,v\rangle}{h_K(x)}:v\in\mathbb{R}^n\right\}.
        \]
        In particular, these eigenvalues have
        multiplicities $1$ and $n$, respectively.
        The next distinct eigenvalue satisfies
        \[
            \lambda_2(-L_K)>1.
        \]

        \item For the unit ball $B$, the operator reduces to 
        $$-L_B=-\frac1{n-1}\Delta_{\mathbb{S}^{n-1}},$$ 
        where $\Delta_{\mathbb{S}^{n-1}}$ is the Laplace-Beltrami operator on $\mathbb{S}^{n-1}$. Then the eigenvalue on spherical harmonics of degree $l$ is $\frac{l(l+n-2)}{n-1}$ for $l=0,1,2,\ldots.$ In particular,
        \[
            \lambda_2(-L_B)=\frac{2n}{n-1}.
        \]

        \item \label{item-4}  If $h_{K_m}\to h_K$ in $C^2$ and $K_m,K\in\mathcal K_+^2$, then
        \[
            \lambda_2(-L_{K_m})
            \longrightarrow
            \lambda_2(-L_K).
        \]

         \item \label{item-5} For any $T\in \text{GL}(n)$, $L_K$ and $L_{T(K)}$ are conjugates via an isometry of Hilbert spaces, and then have the same spectrum
        \begin{align*}
            \sigma(-L_K)=\sigma(-L_{T(K)}).
        \end{align*}
    \end{enumerate}
\end{thm}

In Part~\eqref{item-4}, since the eigenvalues $0$ and $1$ have fixed multiplicities $1$ and $n$, the continuity of $\lambda_2$ follows from the uniform convergence of the weights and coefficients in \eqref{self-ad} and the min-max principle.
In Part~\eqref{item-5}, the equivalence is induced by the map $F_T(x)=\frac{T^{-t}x}{|T^{-t}x|}$. Since $F_T$ is odd, this equivalence preserves even and odd functions. Thus, for origin-symmetric bodies, the even and odd spectral gaps are also invariant under linear maps.

The variational characterization of $\lambda_2(-L_K)$ is
\begin{equation}\label{eq:minmax}
    \lambda_2(-L_K)
    =\inf_{\substack{0\ne z\in H^1(\mathbb{S}^{n-1})\\
    z\perp E_0\oplus E_1^K }}
    \frac{\mathcal E_K(z,z)}{\int z^2\,dV_K},
\end{equation}
where orthogonality is understood in $L^2(V_K)$. The infimum is unchanged if restricted to $C^2$ functions.
To express this characterization without orthogonality constraints, decompose
\begin{equation*}
     z=\bar z+\Pi_Kz+z_\perp,\qquad z_\perp\perp E_0\oplus E_1^K,
\end{equation*}
where $\bar z$ and $\Pi_Kz$ are the orthogonal projections onto $E_0$ and $E_1^K$, respectively. Explicitly,
\begin{equation*}
    \bar z=\frac1{V(K)}\int z\,dV_K.
\end{equation*}
Denoting $\Pi_Kz=\ell_{v_z}^K$, the projection condition gives
\begin{equation*}
    \int\left(z-\frac{\langle x,v_z\rangle}{h_K}\right)\frac{x}{h_K}\,dV_K=0.
\end{equation*}
Let $M_K=\int x\otimes x\,\frac{dV_K}{h_K^2}$, which is positive definite. We have
\begin{equation}\label{eq:translation-projection}
    v_z=M_K^{-1}\int z\,\frac{x}{h_K}\,dV_K,
    \qquad
    \Pi_Kz(x)=\ell_{v_z}^K(x)=\frac{\langle x,v_z\rangle}{h_K(x)}.
\end{equation}
Since $E_0$ and $E_1^K$ are eigenspaces with eigenvalues $0$ and $1$, respectively, \eqref{self-ad} and the orthogonal decomposition yield
\begin{align*}
    \mathcal E_K(z,z)=\int|\Pi_Kz|^2\,dV_K+\mathcal E_K(z_\perp,z_\perp),
\end{align*}
and
\begin{align*}
    \int z^2\,dV_K=V(K)\bar z^2+\int|\Pi_Kz|^2\,dV_K+\int z_\perp^2\,dV_K.
\end{align*}
Consequently, \eqref{eq:minmax} is equivalent to
\begin{equation}\label{eq:minmax-projection}
    \lambda_2(-L_K)
    = \inf_{\substack{ z\in H^1(\mathbb{S}^{n-1})\\ z\notin E_0\oplus E_1^K}}  \frac{\mathcal E_K(z,z)-\int|\Pi_Kz|^2\,dV_K}{\int z^2\,dV_K-V(K)\bar z^2-\int|\Pi_Kz|^2\,dV_K}.
\end{equation}

If $K\in\mathcal K_{+,e}^2$, the operator preserves even and odd functions. Define
\begin{align}
    \lambda_{1,e}(-L_K)
    &=\inf_{\substack{ 0\ne z\in H^1(\mathbb{S}^{n-1})\\ z\text{ even},\ z\perp E_0 }}
    \frac{\mathcal E_K(z,z)}{\int z^2\,dV_K},  \label{eq:even-gap}\\
    \lambda_{2,o}(-L_K)
    &=\inf_{\substack{ 0\ne z\in H^1(\mathbb{S}^{n-1})\\  z\text{ odd},\ z\perp E_1^K }}
    \frac{\mathcal E_K(z,z)}{\int z^2\,dV_K}. \label{eq:odd-gap}
\end{align}
Since $E_1^K$ consists of odd functions, we have
\begin{equation*}
     \lambda_2(-L_K) =\min\{\lambda_{1,e}(-L_K),\lambda_{2,o}(-L_K)\}>1.
\end{equation*}
For the unit ball, these eigenvalues correspond to spherical harmonics of degrees $2$ and $3$, respectively:
\[
    \lambda_{1,e}(-L_B)=\frac{2n}{n-1}, \qquad \lambda_{2,o}(-L_B)=\frac{3(n+1)}{n-1}.
\]

In \cite{KM22}, Kolesnikov and Milman established a significant connection between the even spectral gap and the local $L_p$-Brunn-Minkowski conjecture as follows.
\begin{prop}[\cite{KM22}]  
    For $K\in\mathcal{K}_{+,e}^2$ and $p<1$, the local $L_p$-Brunn-Minkowski conjecture for $K$ is equivalent to the following spectral-gap estimate  
    \begin{align}  
        \lambda_{1,e}(-L_K) \geq \frac{n-p}{n-1}.  
    \end{align}  
\end{prop}  

They also proved that there exists a universal constant $c>0$ such that, with $p_0 = 1 - \frac{c}{n^{\frac{3}{4}}}$,
$$\lambda_{1,e}(-L_K) \geq \frac{n-p_0}{n-1} \qquad\text{for all }K\in\mathcal{K}_{+,e}^2.$$
Combined with the local-to-global principle derived by Chen-Huang-Li-Liu \cite{CHLL20}, this implies that the $L_p$-Brunn-Minkowski conjecture holds for $p \in [p_0, 1)$; see also \cite{Pu21}.

In the opposite direction, Milman \cite{Mi24} proved
\[
    \lambda_{1,e}(-L_K)\le \frac{2n}{n-1},
\]
with equality if and only if $K$ is an origin-centred ellipsoid.

\subsection{Local Brunn-Minkowski inequality}$\ $

The local Brunn-Minkowski inequality is an infinitesimal form of the classical Brunn-Minkowski inequality. Its spectral interpretation, which originates from Hilbert's work, has been further studied in \cites{KM22,Mi25,IM23}.

\begin{lem}[\cites{An97, ACGL20}]\label{key-lem}
    Let $K\in\mathcal K_+^2$ and let $f \in C^2(\mathbb{S}^{n-1})$ satisfy $\int_{\mathbb{S}^{n-1}} f dV_K = 0$. Then
    \begin{align}\label{loc-BM}
        \int_{\mathbb{S}^{n-1}} f^2 dV_K \leq \mathcal E_K(f,f) =\frac{1}{n-1} \int h_K\left\langle (D^2h_K)^{-1}\nabla f,\nabla f\right\rangle dV_K.
    \end{align}
     Equality holds if and only if $f=\ell_v^K$ for some $v\in\mathbb{R}^n$.
\end{lem}

The next eigenvalue yields a refined form that keeps track of the translation component.

\begin{lem}\label{lem-spectral-projection}
    Let $K\in\mathcal K_+^2$ and $f\in H^1(\mathbb{S}^{n-1})$. Then
    \begin{equation}\label{eq:spectral-projection}
        \mathcal E_K(f,f)
        \ge \int|\Pi_Kf|^2\,dV_K+\lambda_2(-L_K)\int|f-\bar f-\Pi_Kf|^2\,dV_K.
    \end{equation}
    For origin-symmetric $K$ and odd $f$, we have $\bar f=0$ and may replace $\lambda_2(-L_K)$ by $\lambda_{2,o}(-L_K)$.
\end{lem}

\begin{proof}
    Decompose $f=\bar f+\Pi_Kf+f_\perp$, where $f_\perp\perp E_0\oplus E_1^K$. Since constants and elements of $E_1^K$ have eigenvalues $0$ and $1$, respectively, we obtain
    \begin{equation*}
        \mathcal E_K(f,f) =\int|\Pi_Kf|^2\,dV_K+\mathcal E_K(f_\perp,f_\perp).
    \end{equation*}
    The claim follows from \eqref{eq:minmax}, and from \eqref{eq:odd-gap} in the odd case.
\end{proof}

An important choice of test functions is given by the Euclidean components of the inverse Gauss map. Let $\{E_l\}_{l=1}^n$ be an orthonormal basis of $\mathbb{R}^n$, and set
\[
    c_K=\frac1{V(K)}\int X_K\,dV_K, \qquad  f_l=\langle X_K-c_K,E_l\rangle.
\]
Applying Lemma~\ref{key-lem} to $\{f_l\}_{l=1}^n$ and summing yields the following inequality.

\begin{lem}[\cite{IM23}]\label{lem-XK}
    Let $K\in\mathcal K_+^2$. Then
    \begin{equation*}
        \int|X_K|^2\,dV_K-\frac{\left|\int X_K\,dV_K\right|^2}{V(K)}
        \le \int h_K\left(\frac1{n-1}\Delta h_K+h_K\right)dV_K.
    \end{equation*}
    Equality holds if and only if $K$ is an origin-centred ellipsoid.
\end{lem}

\begin{rem}  
    The equality condition for the above inequality can be found in Remark 3.4 and the proof of Theorem 1.2 in \cite{IM23}.  
\end{rem}  

By constructing new test functions, one may employ Lemma \ref{key-lem} to derive more general forms of inequalities. This approach proves particularly useful in establishing uniqueness results for prescribed measure problems. For further developments, see, e.g., \cites{CH25,IM24,LW24,HI2401}.


\section{Uniqueness under a spectral condition}\label{sec:3}

\subsection{Main lemmas}\

We first recall an angular second moment identity for solutions to \eqref{is-p-M-eq}.

\begin{lem}[{\cite[Lemma 4.2]{HI2408}}]\label{lem:affine}
Let $K\in\mathcal K_+^2$ satisfy \eqref{is-p-M-eq} with $p\ne-n$. Then
    \begin{equation}\label{eq:affine}
        \int x\otimes x\,dV_K=\frac{V(K)}n I.
    \end{equation}
\end{lem}

\begin{proof}
By the divergence theorem, we have
\begin{equation}\label{eq:flux}
    \int\frac{x}{h_K}\,dV_K=0, \qquad
    \int X_K\otimes\frac{x}{h_K}\,dV_K=\frac{V(K)}n I.
\end{equation}
Indeed, for any $v,w\in\mathbb R^n$,
\begin{align*}
    \int \frac{\langle X_K,v\rangle\langle x,w\rangle}{h_K}\,dV_K
    =& \frac{1}{n}\int_{\partial K}\langle y,v\rangle\langle\nu_K(y),w\rangle\,dS(y)\\
    =&\frac{1}{n}\int_K \operatorname{div}\big(\langle y,v\rangle w\big)\,dy =\frac{V(K)}n\langle v,w\rangle.
\end{align*}
The first identity follows similarly from $\int_{\partial K}\nu_K\,dS=0$.

Equation \eqref{is-p-M-eq} gives $dV_K=h_K^p\,d\mu/n$. Fix $v\in\mathbb R^n$ and write $\ell(x)=\langle x,v\rangle$. Since $\langle X_K,v\rangle=h_K\ell+\langle\nabla h_K,\nabla\ell\rangle$, the second identity in \eqref{eq:flux} yields
\begin{align*}
    \frac{V(K)}{n}|v|^2 = \int\ell^2\,dV_K+\frac{1}{n}\int \ell\langle\nabla h_K,\nabla\ell\rangle h_K^{p-1}\,d\mu.
\end{align*}
If $p\ne0$, integration by parts gives
\begin{align*}
    \frac{1}{n}\int\ell\langle\nabla h_K,\nabla\ell\rangle h_K^{p-1}\,d\mu
    =&\frac{1}{2np}\int\langle\nabla(h_K^p),\nabla(\ell^2)\rangle\,d\mu\\
    =&-\frac{1}{2p}\int\Delta(\ell^2)\,dV_K =-\frac{V(K)}p|v|^2+\frac np\int\ell^2\,dV_K,
\end{align*}
where we used $\Delta(\ell^2)=2|v|^2-2n\ell^2$. Consequently,
\begin{equation*}
    (n+p)\left(\int\ell^2\,dV_K-\frac{V(K)}n|v|^2\right)=0.
\end{equation*}
Since $p\ne-n$, this proves \eqref{eq:affine}. If $p=0$, then $dV_K=d\mu/n$, and the conclusion follows directly from the rotational symmetry of $\mu$.
\end{proof}

Let $\{E_l\}_{l=1}^n$ be an orthonormal basis of $\mathbb R^n$, and consider the test functions
\begin{equation*}
    \widetilde f_l(x)=\langle X_K(x),E_l\rangle, \qquad x\in\mathbb S^{n-1},\quad l=1,\ldots,n.
\end{equation*}
Define their componentwise projection onto $E_1^K$ by
\begin{equation*}
    P(x)=\sum_{l=1}^n(\Pi_K\widetilde f_l)(x)E_l.
\end{equation*}
By \eqref{eq:translation-projection} and \eqref{eq:flux}, we obtain
\begin{equation}\label{eq:XK-projection}
    P(x)=\frac{V(K)}nM_K^{-1}\frac{x}{h_K(x)}.
\end{equation}

\begin{lem}\label{lem-key1}
Let $K\in\mathcal K_+^2$. Then
\begin{equation}\label{ineq-XK-sta}
\begin{split}
    &\int h_K\left(\frac{1}{n-1}\Delta h_K+h_K\right)dV_K\\
    \geq & \int|P|^2\,dV_K +\lambda_2(-L_K)\left(\int|X_K-P|^2\,dV_K-\frac{\left|\int X_K\,dV_K\right|^2}{V(K)}\right).
\end{split}
\end{equation}
If $K$ is origin-symmetric, $\lambda_2(-L_K)$ may be replaced by $\lambda_{2,o}(-L_K)$.
\end{lem}

\begin{proof}
Applying Lemma~\ref{lem-spectral-projection} to $\widetilde f_l$ and summing over $l$, we obtain
\begin{align*}
    \sum_{l=1}^n\mathcal E_K(\widetilde f_l,\widetilde f_l)
    \ge \int|P|^2\,dV_K+\lambda_2(-L_K)\left(\int|X_K-P|^2\,dV_K-\frac{\left|\int X_K\,dV_K\right|^2}{V(K)}\right),
\end{align*}
where $\int P\,dV_K=0$ because $E_1^K\perp E_0$.

To evaluate the left-hand side, let $\{e_i\}_{i=1}^{n-1}$ be a local orthonormal tangent frame. Since $\bar\nabla_{e_i}X_K =\sum_j(D^2h_K)_{ij}e_j$, we have
$$ \sum_{l=1}^n(\widetilde f_l)_i(\widetilde f_l)_j=\sum_k(D^2h_K)_{ik}(D^2h_K)_{jk}.$$
It follows from \eqref{self-ad} that
\begin{align*}
    \sum_{l=1}^n\mathcal E_K(\widetilde f_l,\widetilde f_l)
    =\frac{1}{n-1}\int h_K\tr(D^2h_K)\,dV_K
    =\int h_K\left(\frac{1}{n-1}\Delta h_K+h_K\right)dV_K.
\end{align*}
This proves \eqref{ineq-XK-sta}. If $K$ is origin-symmetric, each $\widetilde f_l$ is odd, so the odd version of Lemma~\ref{lem-spectral-projection} applies.
\end{proof}

\subsection{Proof of Theorem~\ref{thm-unique}}

\begin{proof}[Proof of Theorem~\ref{thm-unique}]
Denote $h=h_K$, $X=X_K$, $V=V(K)$, and set
$$a=\frac{-p-1}{n-1}>1,\qquad \lambda=\lambda_2(-L_K).$$
By \eqref{eq:affine} and \eqref{eq:flux}, we have
\begin{align*}
    \int\nabla h\otimes\frac{x}{h}\,dV_K =\int X\otimes\frac{x}{h}\,dV_K -\int x\otimes x\,dV_K=0.
\end{align*}
Thus every Euclidean component of $\nabla h$ is orthogonal to $E_1^K$, and $X$ and $hx$ have the same componentwise projection $P$. In particular,
\begin{equation*}
    \int\langle\nabla h,P\rangle\,dV_K=0.
\end{equation*}
Define $D=\int|hx-P|^2\,dV_K$. Then
\begin{equation}\label{eq:radial-projection}
    \int h^2\,dV_K=\int|P|^2\,dV_K+D.
\end{equation}
Moreover, since $\langle\nabla h,hx\rangle=0$, we have
\begin{equation}\label{eq:X-res}
    \int|X-P|^2\,dV_K=\int|\nabla h|^2\,dV_K+D.
\end{equation}

We next estimate the mean of $X$ in terms of $D$. For any $u\in C^1(\mathbb S^{n-1})$, spherical integration by parts gives
$$\int\nabla u\,d\mu=(n-1)\int ux\,d\mu.$$
Applying this identity to $u=h^{p+1}$ and using $dV_K=h^p\,d\mu/n$, we obtain
\begin{equation*}
    (p+1)\int\nabla h\,dV_K =(n-1)\int hx\,dV_K.
\end{equation*}
It follows that
\begin{align*}
    \int hx\,dV_K=-a\int\nabla h\,dV_K,\qquad
    \int X\,dV_K=\frac{a-1}{a}\int hx\,dV_K.
\end{align*}
Since $\int P\,dV_K=0$, the Cauchy-Schwarz inequality yields
\begin{equation}\label{eq:mean-bound}
    \frac{\left|\int X\,dV_K\right|^2}{V}
    =\frac{(a-1)^2}{a^2}\frac{\left|\int(hx-P)\,dV_K\right|^2}{V}
    \le\frac{(a-1)^2}{a^2}D.
\end{equation}

On the other hand, integration by parts gives
\begin{align*}
    \int h\Delta h\,dV_K
    &=\frac{1}{n}\int h^{p+1}\Delta h\,d\mu=-\frac{p+1}{n}\int h^p|\nabla h|^2\,d\mu\\
    &=-(p+1)\int|\nabla h|^2\,dV_K.
\end{align*}
Substituting this identity and \eqref{eq:X-res} into Lemma~\ref{lem-key1}, we have
\begin{align*}
    \int h^2\,dV_K+a\int|\nabla h|^2\,dV_K
    \ge  \int|P|^2\,dV_K+ \lambda\left(\int|\nabla h|^2\,dV_K+D -\frac{\left|\int X\,dV_K\right|^2}{V}\right).
\end{align*}
Using \eqref{eq:radial-projection} and \eqref{eq:mean-bound}, we obtain
\begin{equation}\label{eq:deficit}
    (\lambda-a)\int|\nabla h|^2\,dV_K +\left(\lambda\frac{2a-1}{a^2}-1\right)D\le0.
\end{equation}
By the assumption $\lambda\ge a>1$, it follows that
$$\lambda-a\ge0,\qquad \lambda\frac{2a-1}{a^2}-1\ge\frac{a-1}{a}>0.$$
Consequently, $D=0$, which implies $hx=P$ on $\mathbb S^{n-1}$. By \eqref{eq:XK-projection}, this means
$$ h(x)^2x=\frac VnM_K^{-1}x,\qquad x\in\mathbb S^{n-1}.$$
Thus every direction is an eigenvector of the fixed symmetric matrix $M_K^{-1}$. This forces $M_K^{-1}$ to be a scalar multiple of the identity, and hence $h$ is constant. Substituting into \eqref{is-p-M-eq} gives $h^{n-p}=1$, so $h=1$.

If $K$ is origin-symmetric, use the odd version of Lemma~\ref{lem-key1} and set $\lambda=\lambda_{2,o}(-L_K)$. The same argument applies and proves the remaining assertion.
\end{proof}

\section{A geometric lower bound for the spectral gap}
\label{sec:gap}

We adapt the Neumann problem and Reilly formula method of Kolesnikov and Milman \cite[Chapter 6]{KM22}. We apply this method to functions orthogonal to constants and the translation eigenspace. The Payne-Weinberger inequality gives the estimate in terms of concentric inner and outer ellipsoids, while the dimension-free Poincar\'e bound gives \eqref{eq:kls-gap}. We then compare the spectral gaps at different origins.

\begin{proof}[Proof of Theorem~\ref{thm-gap}]
We first assume that $K$ has smooth boundary with $rB\subset K\subset RB$. Let $z$ be a smooth function satisfying the orthogonality conditions in \eqref{eq:minmax}. On $\partial K$, define
$$ \psi=(h_Kz)\circ\nu_K,\qquad s(y)=\langle y,\nu_K(y)\rangle.$$
Since $dV_K=\frac1n h_K\,dS_K =\frac1n s\,dS$ under the inverse Gauss parametrization, the orthogonality conditions imply
\begin{equation}\label{eq:compat}
    \int_{\partial K}\psi\,dS=0,\qquad
    \int_{\partial K}\frac{\psi\langle v,\nu_K\rangle}{s}\,dS=0
    \quad\text{for all }v\in\mathbb R^n.
\end{equation}
Let $u$ be a solution of the Neumann problem
\begin{align*}
    \bar\Delta u=0\quad\text{in }K,\qquad
    \partial_{\nu_K}u=\psi\quad\text{on }\partial K.
\end{align*}
The first condition in \eqref{eq:compat} is precisely the compatibility condition for this problem.

Let $\mathcal S$ denote the positive shape operator of $\partial K$ and let $H=\operatorname{tr}\mathcal S$ be its mean curvature. Since $u$ is harmonic, Reilly's formula \cite{Reilly77} gives
\begin{equation}\label{eq:reilly-boundary}
\begin{split}
     \int_K|\bar\nabla^2u|^2\,dy
    =&\int_{\partial K}\left( 2\left\langle\nabla_{\partial K}\psi,\nabla_{\partial K}u\right\rangle -H\psi^2-\left\langle\mathcal S\nabla_{\partial K}u, \nabla_{\partial K}u\right\rangle \right)dS \\
    \leq& \int_{\partial K}\left( \left\langle\mathcal S^{-1}\nabla_{\partial K}\psi,\nabla_{\partial K}\psi\right\rangle-H\psi^2 \right)dS.
\end{split}
\end{equation}
Under the inverse Gauss parametrization, set $f=h_Kz$. Then
$$\psi\circ X_K=f,\qquad \mathcal S^{-1}\circ X_K=D^2h_K,\qquad (\nabla_{\partial K}\psi)\circ X_K=(D^2h_K)^{-1}\nabla f,$$
and
$$H=\tr\big((D^2h_K)^{-1}\big),\qquad dS=\det(D^2h_K)\,d\mu.$$
Therefore, the right-hand side of \eqref{eq:reilly-boundary} equals
\begin{align*}
    &\int_{\mathbb S^{n-1}}\left(\left\langle(D^2h_K)^{-1}\nabla f,\nabla f\right\rangle-f^2\tr\big((D^2h_K)^{-1}\big)\right)\det(D^2h_K)\,d\mu\\
    =& -\int_{\mathbb S^{n-1}} f\,\tr\big((D^2h_K)^{-1}D^2f\big)\det(D^2h_K)\,d\mu,
\end{align*}
where we used the divergence-free property of the cofactor tensor of $D^2h_K$. By the definition of $L_K$, we have
$$\tr\big((D^2h_K)^{-1}D^2f\big)=(n-1)(L_Kz+z).$$ 
Using $f=h_Kz$ and $dV_K=\frac{1}{n} h_K\det(D^2h_K)\,d\mu$, we conclude that
\begin{equation}\label{eq:reilly}
    \int_K|\bar\nabla^2u|^2\,dy \le n(n-1)\left(\int z(-L_Kz)\,dV_K -\int z^2\,dV_K\right).
\end{equation}

Define $W=\bar\nabla u-v$ such that $\int_KW\,dy=0$. Then the vector $v$ is given by
$$v=\frac1{V(K)}\int_K\bar\nabla u\,dy.$$
On $\partial K$, we have
\begin{align*}
     \langle W,\nu_K\rangle
    =\partial_{\nu_K}u-\langle v,\nu_K\rangle
    =\psi-\langle v,\nu_K\rangle.
\end{align*}
The second condition in \eqref{eq:compat} then implies
\begin{equation}\label{eq:trace-orth}
\begin{split}
     n\int z^2\,dV_K
    =&\int_{\partial K}\frac{\psi^2}{s}\,dS =\int_{\partial K}\frac{\big(\psi-\langle v,\nu_K\rangle\big)^2}{s}\,dS
    -\int_{\partial K}\frac{\langle v,\nu_K\rangle^2}{s}\,dS\\
    \leq& \int_{\partial K}\frac{\langle W,\nu_K\rangle^2}{s}\,dS
    \leq \frac{1}{r^2}\int_{\partial K}|W|^2s\,dS,
\end{split}
\end{equation}
where we used $s\ge r$ and $|\langle W,\nu_K\rangle|\le |W|$. Since $\operatorname{div} W=\Delta u=0$, we have
\begin{equation*}
    \int_{\partial K}|W|^2s\,dS =\int_K\operatorname{div}(|W|^2y)\,dy
    =n\int_K|W|^2\,dy+2\int_K\langle W,\bar\nabla^2u\,y\rangle\,dy.
\end{equation*}
Since $|y|\le R$ in $K$, the Cauchy-Schwarz inequality yields
\begin{equation*}
    \int_{\partial K}|W|^2s\,dS \leq n\int_K|W|^2\,dy+2R\left(\int_K|W|^2\,dy\right)^{\frac{1}{2}}\left(\int_K|\bar\nabla^2u|^2\,dy\right)^{\frac{1}{2}}.
\end{equation*}
Let $C_{\mathrm P}(K)$ be the least constant in the Neumann Poincar\'e inequality
\[
    \int_K f^2\,dy\le C_{\mathrm P}(K)\int_K|\bar\nabla f|^2\,dy,
    \qquad \int_K f\,dy=0.
\]
Applying this inequality to each component of $W$ gives, by \eqref{eq:trace-orth} and \eqref{eq:reilly},
\begin{equation*}
    \int z(-L_Kz)\,dV_K-\int z^2\,dV_K
    \geq \frac{r^2}{(n-1)\big(n C_{\mathrm P}(K)+2R\sqrt{C_{\mathrm P}(K)}\big)} \int z^2\,dV_K.
\end{equation*}
Taking the infimum over all admissible functions $z$ gives
\begin{equation}\label{eq:gap-poincare}
    \lambda_2(-L_K)\geq 1+\frac{r^2}{(n-1)\big(n C_{\mathrm P}(K)+2R\sqrt{C_{\mathrm P}(K)}\big)} .
\end{equation}

For the explicit bound in terms of $r$ and $R$, since $\operatorname{diam}(K)\le 2R$, the Payne-Weinberger inequality \cite{PW60} gives
\[
    C_{\mathrm P}(K)\le \frac{4R^2}{\pi^2}.
\]
Therefore, we obtain
\[
    n C_{\mathrm P}(K)+2R\sqrt{C_{\mathrm P}(K)}
    \leq \left(\frac{4n}{\pi^2}+\frac4\pi\right)R^2=A_nR^2.
\]
Inserting this estimate into \eqref{eq:gap-poincare}, we conclude that
\[
    \lambda_2(-L_K)\geq 1+\frac{1}{(n-1)A_n(R/r)^2}.
\]

The radius estimate extends to $K\in\mathcal K_+^2$ by smoothing the support function in $C^2$ and using Theorem~\ref{thm-LK}\eqref{item-4}. The inner and outer radii can be chosen to converge under the approximation. By the linear equivariance of the spectrum and the definition of $\kappa_0(K)$, we may optimize over all $T\in\mathrm{GL}(n)$ and obtain
\begin{equation*}
     \lambda_2(-L_K)\geq 1+\frac{1}{(n-1)A_n\kappa_0(K)^2}.
\end{equation*}

We next assume that the centroid of $K$ is at the origin. By linear equivariance, we may further assume that $K$ is in isotropic position, i.e.,
\[
    \frac{1}{V(K)}\int_K y\otimes y\,dy=I.
\]
Theorem~4.1 of Kannan, Lov\'asz and Simonovits \cite{KLS95} then gives
\[
    \sqrt{\frac{n+2}{n}}\,B\subset T(K)\subset \sqrt{n(n+2)}\,B
\]
for some $T\in\mathrm{GL}(n)$, and hence $\kappa_0(K)\le n$. In addition, the dimension-free Poincar\'e theorem gives $C_{\mathrm P}(K)\le25$ \cite{BK26}. Thus, \eqref{eq:gap-poincare} and the above inertia bound imply
\begin{equation*}
    \lambda_2(-L_K)\geq 1+\frac{\eta_n}{n-1}.
\end{equation*}

If $K$ is origin-symmetric, the symmetric form of John's theorem \cite[Section~10.12]{Sch14} gives the sharper estimate $\kappa_0(K)\le\sqrt n$. This completes the proof.
\end{proof}

We next show the effect of changing the origin to the centroid.

\begin{lem}\label{lem:origin-gap}
Let $K\in\mathcal K_+^2$, and let $c$ be its centroid. If $h_K\ge\beta h_{K-c}$ for some $\beta>0$, then
\begin{equation}\label{eq:origin-gap}
    \lambda_2(-L_K)-1\ge\beta\big(\lambda_2(-L_{K-c})-1\big).
\end{equation}
\end{lem}

\begin{proof}
We first assume that $\partial K$ is smooth. For a function $f$ on $\partial K$, let
\begin{align*}
    \mathcal Q_K(f):= \frac{\int_{\partial K}\bigl(\langle\mathcal S^{-1}\nabla_{\partial K}f,\nabla_{\partial K}f\rangle-Hf^2\bigr)\,dS}{(n-1)\int_{\partial K}\frac{f^2}{\langle y,\nu_K\rangle}\,dS}.
\end{align*}
As in the calculation leading to \eqref{eq:reilly}, this is the Rayleigh quotient of $-L_K-1$. Since $-L_K$ has one eigenvalue equal to $0$ and an $n$-dimensional eigenspace at $1$, the first positive eigenvalue of $-L_K-1$ is characterized by
\begin{equation*}
    \lambda_2(-L_K)-1=\inf_{\substack{E\subset H^1(\partial K)\\ \dim E=n+2}}\, \sup_{0\ne f\in E}\mathcal Q_K(f).
\end{equation*}

We identify $\partial(K-c)$ with $\partial K$ by the translation $y\mapsto y-c$. Under this identification, the surface measure, the shape operator, the mean curvature, and the tangential derivatives are unchanged. Hence, the numerator of $\mathcal Q_K$ is the same as the numerator of $\mathcal Q_{K-c}$. On the other hand, we have
\[
\langle y,\nu_K(y)\rangle=h_K(\nu_K(y)),
\qquad
\langle y-c,\nu_K(y)\rangle=h_{K-c}(\nu_K(y)).
\]
The assumption $h_K\ge\beta h_{K-c}$ then gives
\begin{align*}
    \int_{\partial K}\frac{f^2}{\langle y,\nu_K\rangle}\,dS
    \leq \int_{\partial K}\frac{f^2}{\beta\langle y-c,\nu_K\rangle}\,dS
    =\frac{1}{\beta}\int_{\partial (K-c)}\frac{f^2}{\langle y,\nu_{K-c}\rangle}\,dS.
\end{align*}
Thus, whenever the numerator of $\mathcal Q_K(f)$ is positive,
\[
\mathcal Q_K(f)\geq\beta\,\mathcal Q_{K-c}(f).
\]
Taking the infimum over all subspaces $E$ of dimension $n+2$ yields \eqref{eq:origin-gap}.

For a general $K\in\mathcal K_+^2$, choose smooth strictly convex bodies $K_j$ such that $h_{K_j}\to h_K$ in $C^2$, and let $c_j$ denote the centroid of $K_j$. Then $c_j\to c$, and hence $h_{K_j-c_j}\to h_{K-c}$ uniformly. Therefore, the assumption $h_K\geq\beta h_{K-c}$ implies
\[
h_{K_j}\geq(\beta-o(1))h_{K_j-c_j}.
\]
Applying the smooth case to $K_j$ and passing to the limit, using the continuity of the centroid and of $\lambda_2(-L_K)$ under $C^2$ convergence, we complete the proof of \eqref{eq:origin-gap}.
\end{proof}

\section{The interval of uniqueness}\label{sec:interval}

In this section, we combine the spectral estimates from Theorem~\ref{thm-gap} with a weighted centroid identity for the polar body. This yields an explicit interval below the critical exponent $-n$ in which the unit ball is the unique solution.

Let $p=-n-\varepsilon$ ($\varepsilon>0$), and let $K$ be a solution to \eqref{is-p-M-eq}. Denote by $K^\circ$ the polar body of $K$ with respect to the origin. Since $\rho_{K^\circ}=h_K^{-1}$, the first identity in \eqref{eq:flux} gives
\begin{equation}\label{eq:polarcentroid}
    \int_{K^\circ}y|y|^\varepsilon\,dy
    =\frac{1}{n+\varepsilon+1}\int_{\mathbb S^{n-1}}xh_K^{-n-\varepsilon-1}\,d\mu=0.
\end{equation}
Thus, the origin is the weighted centroid of $K^\circ$. We use this identity to compare the origin with the ordinary centroid of a suitable linear image of $K^\circ$.

\begin{lem}\label{lem:polar}
Let $K$ satisfy \eqref{is-p-M-eq} with $p=-n-\varepsilon$. If $d_n\varepsilon<1$, then
    \begin{equation}\label{eq:kappa-sol}
        \kappa_0(K)\le\frac{n}{1-d_n\varepsilon}.
    \end{equation}
\end{lem}

\begin{proof}
Let $L=T(K^\circ)$, where $T\in\mathrm{GL}(n)$ is chosen so that the inertia ellipsoid of $L$, computed about its centroid $z$, is a Euclidean ball. By the estimate of Kannan, Lov\'asz and Simonovits used in the proof of Theorem~\ref{thm-gap}, after a common rescaling we have
\begin{equation}\label{assump-L-z}
    rB\subset L-z\subset RB,\qquad\frac{R}{r}=n.
\end{equation}
Denote $M_0=\max\limits_{y\in L}|T^{-1}y|$ and $w(y)=(\frac{|T^{-1} y|}{M_0})^\varepsilon$.
Since the map $y=T x$ has constant Jacobian, \eqref{eq:polarcentroid} implies
\begin{equation}\label{eq-weighted}
    \int_L yw(y)\,dy=0.
\end{equation}

We first estimate the sublevel sets of $|T^{-1}y|$. Choose $y_0\in L$ such that $|T^{-1}y_0|=M_0$, and set $e=\frac{T^{-1}y_0}{M_0}$. Consider the linear function $\ell(y)=\langle e,T^{-1} y\rangle$ and the slab
$$S_t = \{y\in L:|\ell(y)|<tM_0\},\qquad 0<t<1.$$
Let $[\alpha,\beta]=\ell(L)$. Since $\ell(0)=0$ and $\ell(y_0)=M_0$, we have $\beta-\alpha\geq M_0$. For $s\in[\alpha,\beta]$, define
$$A(s)=\mathcal H^{n-1}\big(L\cap\ell^{-1}(s)\big), \qquad A_*=\max_{s\in[\alpha,\beta]}A(s).$$
Choose $s_*$ such that $A(s_*)=A_*$. Choose one point in each endpoint section. The two cones joining these points to the maximal section $L\cap\ell^{-1}(s_*)$ lie in $L$ and have disjoint interiors. Their volumes give
\begin{align*}
    V(L)\geq \frac{(\beta-\alpha)A_*}{n|T^{-t}e|}.
\end{align*}
On the other hand, since the range of $\ell$ on $S_t$ is at most $2t M_0$, and every section of $L$ has area at most $A_*$, we have
\begin{align*}
     V(S_t)\le\frac{2tM_0A_*}{|T^{-t}e|}.
\end{align*}
Consequently, we obtain
\begin{equation*}
    \frac{V(S_t)}{V(L)}\leq\frac{2ntM_0}{\beta-\alpha} \leq 2nt.
\end{equation*}
Since $|e|=1$, $|T^{-1}y|<tM_0$ implies $|\ell(y)|<tM_0$. Thus,
\begin{equation}\label{slab-estimate}
    \frac{V\big(\{y\in L:|T^{-1}y|<tM_0\}\big)}{V(L)}\leq \min\{1,2nt\},\qquad 0<t<1.
\end{equation}

We next estimate the difference between the weighted and ordinary centroids. Define
$$\delta=\frac{1}{V(L)}\int_L(1-w(y))\,dy\in (0,1).$$
By \eqref{slab-estimate}, using $1-w(y)=\int_{|T^{-1}y|/M_0}^1\varepsilon t^{\varepsilon-1}\,dt$, we obtain
\begin{equation}\label{eq:delta}
\begin{split}
    \delta =& \frac{1}{V(L)}\int_0^1 \varepsilon t^{\varepsilon-1}V(\{y\in L:\, |T^{-1}y|<t M_0\}) \,dt\\
    &\leq \int_0^1 \min\{1,2nt\}\, \varepsilon t^{\varepsilon-1}\,dt\\
    &=1-\frac{(2n)^{-\varepsilon}}{1+\varepsilon} \leq (1+\log(2n))\varepsilon.
\end{split}
\end{equation}
Since $d_n\varepsilon<1$, we have $\delta<\frac{1}{n+1}$. Recall that $z$ is the ordinary centroid of $L$. It follows from \eqref{eq-weighted} that
\begin{align*}
     z=\frac{1}{V(L)}\int_L y\,dy=\frac{1}{V(L)}\int_L y(1-w(y))\,dy.
\end{align*}
Hence, for every $x\in\mathbb S^{n-1}$,
\begin{align*}
    -\delta h_L(-x)\leq \langle z,x\rangle\leq \delta h_L(x).
\end{align*}
Then we have
\begin{equation}\label{L-z-upper}
    h_{L-z}(x)=h_L(x)-\langle z,x\rangle\geq (1-\delta)h_L(x).
\end{equation}
On the other hand, the Minkowski asymmetry estimate in \cite[Section~6.1]{Gru63} gives
\begin{equation*}
    h_{L-z}(-x)\le n h_{L-z}(x), \qquad x\in\mathbb S^{n-1}.
\end{equation*}
Thus,
\begin{align*}
     -\langle z,x\rangle\leq\delta h_L(-x)
     =\delta\big(h_{L-z}(-x)-\langle z,x\rangle\big)
    \le\delta\big(n h_{L-z}(x)-\langle z,x\rangle\big),
\end{align*}
and hence
\begin{equation}\label{L-z-lower}
      h_L(x)=h_{L-z}(x)+\langle z,x\rangle \geq \frac{1-(n+1)\delta}{1-\delta}h_{L-z}(x).
\end{equation}
Combining \eqref{assump-L-z} with \eqref{L-z-upper} and \eqref{L-z-lower}, we conclude that
\begin{equation*}
    r\leq h_{L-z}(x)\leq \frac{1-\delta}{1-(n+1)\delta}h_L(x),\qquad
    R\geq h_{L-z}(x)\geq (1-\delta)h_L(x),
\end{equation*}
and then
\begin{equation*}
     \frac{1-(n+1)\delta}{1-\delta} rB\subset L\subset \frac{1}{1-\delta}RB=\frac{n}{1-\delta} rB.
\end{equation*}
That is, $L$ is contained between concentric balls whose radii have ratio at most $\frac{n}{1-(n+1)\delta}$.

Notice that polarity preserves the ratio of concentric inner and outer
balls. Since $L^\circ=T^{-t}(K)$, we obtain
\begin{align*}
     \kappa_0(K)\leq\frac{n}{1-(n+1)\delta}.
\end{align*}
Therefore, \eqref{eq:kappa-sol} follows from \eqref{eq:delta} and $d_n=(n+1)(1+\log(2n))$. We complete the proof.
\end{proof}

\begin{proof}[Proof of Theorem~\ref{thm-interval}]
Denote $p=-n-\varepsilon$, and put
\[
    \tilde{\varepsilon}_n=\frac{4b_n}{\left(1+\sqrt{1+4d_nb_n}\right)^2},\qquad
    \tilde{\eta}_n=\min\left\{\frac{1}{2d_n},\frac{\eta_n}{2n+1}\right\}.
\]
Thus $\varepsilon_n=\max\{\tilde{\varepsilon}_n,\tilde{\eta}_n\}$. Suppose first that $0<\varepsilon\le\tilde{\varepsilon}_n$. By the definition of $\tilde{\varepsilon}_n$, we have
$$\sqrt{\frac{\tilde{\varepsilon}_n}{b_n}} +d_n\tilde{\varepsilon}_n=1.$$
It follows that $d_n\varepsilon<1$ and $\varepsilon \leq b_n(1-d_n\varepsilon)^2$. Lemma~\ref{lem:polar} and \eqref{eq:geomgap} therefore give
\begin{align*}
    \lambda_2(-L_K)
    &\geq 1+\frac{b_n(1-d_n\varepsilon)^2}{n-1}\\
    &\geq 1+\frac{\varepsilon}{n-1}=\frac{-p-1}{n-1}.
\end{align*}
Theorem~\ref{thm-unique} then implies that $K$ is the unit ball.

Suppose next that $0<\varepsilon\leq\tilde{\eta}_n$. Since $d_n\varepsilon\leq\frac{1}{2}$, Lemma~\ref{lem:polar} gives $\kappa_0(K)\le2n$ and hence $-K\subset 2nK$. Let $c$ be the centroid of $K$. Since $c\in K$,
\[
    h_{K-c}(x)\le h_K(x)+h_K(-x)\le(2n+1)h_K(x).
\]
Lemma~\ref{lem:origin-gap} and \eqref{eq:kls-gap} give
\[
    \lambda_2(-L_K)-1\ge\frac{\lambda_2(-L_{K-c})-1}{2n+1}
    \ge\frac{\eta_n}{(n-1)(2n+1)}\ge\frac{\varepsilon}{n-1}.
\]
Theorem~\ref{thm-unique} again implies that $K$ is the unit ball. These two arguments prove the assertion for $0<\varepsilon\le\varepsilon_n$.

If $K$ is origin-centred, Theorem~\ref{thm-gap} gives $\kappa_0(K)\le n$ and
\[
    \lambda_2(-L_K)-1\ge\frac{\max\{b_n,\eta_n\}}{n-1}.
\]
If $K$ is origin-symmetric, the estimate $\kappa_0(K)\le\sqrt n$ gives
\[
    \lambda_2(-L_K)-1\ge\frac{\max\{nb_n,\eta_n\}}{n-1}.
\]
The corresponding assertions follow from Theorem~\ref{thm-unique}.
\end{proof}


\paragraph{\textbf{AI Declaration}.}
Building on the author's previous version, arXiv:2509.08588v1, and subsequent work, the author used OpenAI's GPT-6.1 Sol to assist with spectral-gap estimates, particularly the estimate in Lemma~\ref{lem:polar}, as well as with exposition and grammar. The author takes full responsibility for the mathematical content and conclusions of this paper.

\end{document}